\documentclass[12pt]{amsart}
\usepackage{amsmath}
\usepackage{amssymb}
\usepackage{latexsym}
\def\ds{\displaystyle}
\newtheorem{theorem}{Theorem}[section]
\newtheorem{cor}[theorem]{Corollary}
\newtheorem{lemma}[theorem]{Lemma}
\newtheorem{exam}[theorem]{Example}
\newtheorem{prop}[theorem]{Proposition}

\newtheorem{defn}[theorem]{Definition}

\newenvironment{proof*}{\vskip 2mm\noindent {}}{\hfill $\Box$ \vskip 2mm}
\numberwithin{equation}{section}
\newcommand{\C}{{\mathbb{C}}}

\newcommand{\D}{{\mathbb{D}}}
\newcommand{\G}{{\mathbb{G}}}
\renewcommand{\O}{{\mathcal{O}}}
\newcommand{\Z}{{\mathbb{Z}}}

\begin{document}

\title{Discontinuity of the Lempert function of the spectral ball}

\author{P. J. Thomas, N. V. Trao}

\address{Universit\'e de Toulouse\\ UPS, INSA, UT1, UTM \\ 
Institut de Math\'ematiques de Toulouse\\
F-31062 Toulouse, France} 
\email{pthomas@math.univ-toulouse.fr}

\address{ Department of Mathematics\\
Hanoi National University of Education\\
136 Xuan Thuy str - Cau Giay\\ 
Hanoi - Vietnam}
\email{ngvtrao@yahoo.com}

\begin{thanks}{The initial version of this paper was written during the
stay of the second named author at the Paul Sabatier University,
Toulouse.}
\end{thanks}

\begin{abstract} 
We give some further criteria for continuity or discontinuity of
 the Lempert funtion of the spectral ball $\Omega_n$, with respect
to one or both of its arguments, in terms of cyclicity the matrices involved.
\end{abstract}

\maketitle

\section{Introduction and statement of results}
The spectral ball is the set of all $n\times n$ 
complex  matrices with eigenvalues strictly
smaller than one in modulus. It can be seen as the union of all the unit
balls of the space of matrices endowed with all the possible operator
norms arising from a choice of norm on the space $\mathbb C^n$. It  contains
many entire curves. As analogues of the Montel theorem cannot hold, several  invariant objects
in complex analysis 
exhibit discontinuity phenomena in this setting, 
first pointed out in \cite{Agl-You1}.

The goal of this note is to give a few facts 
about discontinuities of the Lempert
function (corresponding to the two-point 
Pick-Nevanlinna problem), to be compared with
 \cite{NiThZw}.

We fix some notation.
Let $\mathcal M_n$ be the set of all $n\times n$ complex matrices.
For $A\in\mathcal M_n$ denote by $sp(A)$ and $\ds
r(A)=\max_{\lambda\in sp(A)}|\lambda|$ the spectrum and the
spectral radius of $A,$ respectively. The \emph{spectral ball} $\Omega_n$
is the set
$$\Omega_n=\{A\in\mathcal M_n:r(A)<1\}.$$
The
characteristic polynomial of the matrix $A$ is
$$
P_A(t):= \det(tI-A)=: t^n+\sum_{j=1}^n(-1)^j\sigma_j (A)t^{n-j},
$$
where $I\in \mathcal M_n$ is the unit matrix. We define
a map $\sigma$ from $\mathcal M_n$ to $\mathbb C^n$ by
$\sigma := (\sigma_1, \dots, \sigma_n)$. 
 The \emph{symmetrized polydisk}
is $\mathbb G_n := \sigma (\Omega_n)$ is a bounded domain in $\mathbb C^n$,
which is hyperconvex \cite{EdigarianZwonek} and, therefore,
a complete hyperbolic domain, and taut. As is noted in the same paper
or in \cite[Proposition 7]{NiThZw}, $\sigma(A)=\sigma(B)$ if and only if
there is an entire curve contained in $\Omega_n$ going through $A$ and $B$.

For general facts about invariant 
(pseudo)distances and (pseudo)metrics, see for
instance \cite{JarnickiPflug}.
The Lempert function of a domain $D\subset\Bbb C^m$ is defined, for $z,w\in D$, as
$$
 l_D(z,w):=\inf\{|\alpha|: \alpha \in \mathbb D 
\mbox{ and } \exists\varphi\in\mathcal O(\mathbb
D,D):\varphi(0)=z,\varphi(\alpha)=w\}.
$$
The Lempert function is  upper semicontinuous and
 decreases under holomorphic maps, so for  $A, B \in \Omega_n,$
\begin{equation}
\label{decr}
l_{\Omega_n}(A,B) \ge l_{\mathbb{G}_n}(\sigma(A),\sigma(B)).
\end{equation}
On $\mathbb G_n$, the Lempert function is continuous, and the remark above
about entire curves shows that $l_{\mathbb{G}_n}(\sigma(A),\sigma(B))=0$ if and
only if $ l_{\Omega_n}(A,B) =0$.

For a Zariski dense open set of 
matrices, equality holds in \eqref{decr} and the Lempert
function is continuous. Recall that
a matrix $A$ is \emph{cyclic} (or \emph{non-derogatory})
if it admits a cyclic vector (see for instance 
\cite{HoJo}). As in \cite{NiThZw}, we denote
by $C_{\sigma(A)}$ the companion matrix of the 
characteristic polynomial of $A$; $A$ 
is cyclic
if and only if it is conjugate to $C_{\sigma(A)}$.

Agler and Young \cite{Agl-You1} proved that
if $A$ and $B$ are cyclic, any
holomorphic mapping $\varphi \in \mathcal O(\mathbb D, \mathbb G_n)$
  through $\sigma(A)$ and $\sigma(B)$ lifts to 
$\Phi \in \mathcal O(\mathbb D, \Omega_n)$
through $A$ and $B$, so that in particular:
\begin{prop}[Agler-Young]
\label{lifting}
If $A, B \in \Omega_n$ are cyclic, then
\begin{equation}
\label{equal}
l_{\Omega_n}(A,B) = l_{\mathbb{G}_n}(\sigma(A),\sigma(B)).
\end{equation}
\end{prop}

Continuity of the Lempert function near such a 
pair $(A,B)$ follows from the fact that
cyclicity is an open condition, or from the following.
\begin{prop}
\label{contcrit}
Let $A, B \in \Omega_n$.
\begin{enumerate}
\item
The Lempert function $l_{\Omega_n}$ is continuous at $(A,B)$ if and only if
\eqref{equal} holds.
\item
If $B$ is cyclic, and the function $l_{\Omega_n}(.,B)$ is continuous at $A,$
then \eqref{equal} holds.
\end{enumerate}
\end{prop}
The proofs are given in Section \ref{pfcontcrit}. 

We see that when $B$ is cyclic, 
continuity of the Lempert function
with respect to both variables reduces to continuity with respect to the first
variable. We now study this partial continuity for iself.

\begin{theorem}
\label{continuous}
Let $A \in \Omega_n.$
If $A$  is not cyclic (or \emph{derogatory}), then there exists a matrix $B\in \Omega_n$
such that 
  the function $l_{\Omega_n}(.,B)$ is not continuous at $A$.
\end{theorem}

The proof is given in Section \ref{pfcont}.

Conversely, suppose that $A$ is cyclic. If $B$ is cyclic, Proposition \ref{lifting} settles the
question. If not, the converse holds in several interesting cases. 

\begin{prop}
\label{contlemp}
Let $A, B \in \Omega_n$, $A$ cyclic. If $B$ has only one
eigenvalue, or  $n\le 3$,
then  the function $l_{\Omega_n}(.,B)$ is continuous at $A$.
\end{prop}

This is proved in Section \ref{pfnil}.

We conjecture that the additional hypotheses on $B$ or $n$ can be dispensed with.

When $A$ is derogatory, one may wonder which matrices $B$ make the
function $l_{\Omega_n}(.,B)$ continuous. The case where $A=tI$, treated
in \cite{NiThZw}, suggests that discontinuity is  generic.
We give some relevant examples in Section \ref{exs}. In Section \ref{green}, we 
apply our results to compare the Lempert and Green functions on the spectral ball. 

\section{Proof of Proposition \ref{contcrit}.}
\label{pfcontcrit}

\subsection{Proof of (i), direct part}

Since the cyclic matrices are dense in
$\Omega_n$ then there exist  $A_j , B_j \in 
\mathcal{C}_n$ such that $A_j \to A , B_j \to B.$
By continuity of  $l_{\Omega_n}$ at $(A,B)$ we 
get that $l_{\Omega_n}(A_j,B_j) \xrightarrow{j 
\to \infty} l_{\Omega_n}(A,B).$

On the other hand $l_{\Omega_n}(A_j,B_j) = 
l_{\mathbb{G}_n}(\sigma(A_j),\sigma(B_j)).$
By tautness of the domain $\mathbb{G}_n$ we have 
$l_{\mathbb{G}_n}(\sigma(A_j),\sigma(B_j)) 
\xrightarrow{j \to \infty} 
l_{\mathbb{G}_n}(\sigma(A),\sigma(B)).$
This implies that $l_{\Omega_n}(A,B) = l_{\mathbb{G}_n}(\sigma(A),\sigma(B)).$

\subsection{Proof of (i), converse part}

Assume $l_{\Omega_n}(A,B) = l_{\mathbb{G}_n}(\sigma(A),\sigma(B))$.

Let $(A_j,B_j)\subset \Omega_n$ be such that 
$(A_j,B_j) \xrightarrow{j\to \infty} (A,B)$ and
$$
\lim_{j\to \infty}l_{\Omega_n}(A_j,B_j) = a := 
\liminf_{(X,Y)\to (A,B)} l_{\Omega_n}(X,Y).
$$
We have
$$
l_{\Omega_n}(A_j,B_j) \ge 
l_{\mathbb{G}_n}(\sigma(A_j),\sigma(B_j)) \to 
l_{\mathbb{G}_n}(\sigma(A),\sigma(B)),
$$
and hence $a \geq l_{\mathbb{G}_n}(\sigma(A),\sigma(B)) = l_{\Omega_n}(A,B).$
Then $l_{\Omega_n}$ is lower semicontinuous at $(A,B).$
Since  $l_{\Omega_n}$ is always upper 
semicontinous, it is continuous at $(A,B)$.

\subsection{Proof of (ii).}

We only need to repeat the proof of the direct 
part of (i), taking $B_j=B$ for all $j$.
Then we only use the continuity of $l_{\Omega_n}$ in the first variable.
\qed

\section{ Proof of Theorem \ref{continuous}.}
\label{pfcont}

We shall need a theorem by Bharali \cite{bharali}. 

\begin{theorem}[Bharali]
Let $F\in \O(\D,\Omega_n), n \ge 2,$ and let $\zeta_1, \zeta_2 \in \D.$ 
Write $W_j=F(\zeta_j), j=1, 2.$  If $\lambda \in sp(W_j),$ then let $m(\lambda)$ 
denote the multiplicity of $\lambda$ as a zero of the minimalpolynomial
of $W_j.$ Then
\begin{multline}
\label{equationofBharali}
\max\bigg\{\!
\max_{\mu \in sp(W_2)}\!\prod_{\lambda \in sp(W_1)}\!\bigg| \frac{\mu-\lambda}{1-\overline{\lambda}\mu}\bigg|^{m(\lambda)};\!
\max_{\lambda \in sp(W_1)}\!\prod_{\mu \in sp(W_2)}\!\bigg| \frac{\lambda - \mu}{1-\overline{\mu}\lambda}\bigg|^{m(\mu)}
\!\bigg\}\\
\le \bigg| \frac{\zeta_1 - \zeta_2}{1-\overline{\zeta_2}\zeta_1}\bigg|.
\end{multline}
\end{theorem}

Now, let $A$ be derogatory.
The idea will be to construct a matrix $B$ a short distance away
from $A$, in a direction which belongs to the kernel of the differential
map of $\sigma$ at $A$, but where the Kobayashi Royden pseudometric
doesn't vanish.  Compare with the proof of Proposition 3 and in
particular Lemma 8 in \cite{zerosetofkobayashi}.

Since $A$ is derogatory, at least two of the
eigenvalues of $A$ are equal, say to $\lambda.$ Applying the
automorphism of $\Omega_n$ given by $M\mapsto 
(\lambda I - M)(I - \bar \lambda M)^{-1}$,
  we may assume that
$\lambda=0.$ Since the map $A \to P^{-1} A P$ is a linear
automorphism of $\Omega_n$ for any $P\in\mathcal M_n^{-1}$, we may
also assume that $A$ is in Jordan form. In particular,
$$
A = \left( \begin{array}{cc} A_0 & 0 \\ 0 & A_1 \end{array} \right),
$$
where $A_0 \in \mathcal M_m$, $2\le m \le n$, $sp(A_0)=\{0\}$, $A_1
\in \mathcal M_{n-m}$, $0 \notin sp(A_1)$. Furthermore, there is a
set $J \subsetneq \{2, \dots, m\}$, possibly empty, such that
$a_{j-1,j}=1$ for $j \in J$, and all other coefficients $a_{ij}=0$
for $1\le i, j \le m$. Denote $0\le r := \# J = \mbox{rank} A_0 \le
m-2$ and $k$ is the multiplicity of $0$ as a zero of the minimal
polynomial of $A_0, 0<k \le r+1 < m.$

We set
$$
X := \left( \begin{array}{cc} X_0 & 0 \\ 0 & 0 \end{array}
\right) \in \mathcal M_n,
$$
where $X_0 = (x_{ij})_{1\le i, j \le m}$
is such that $x_{j-1,j}=-1$ for $j \in \{2, \dots, m\} \setminus J$,
$x_{m1}=1$, and $x_{ij}=0$ otherwise.

For $\delta >0$ is small enough we set
$$
B=A+\delta X = \left( \begin{array}{cc} B_0 & 0 \\ 0 & A_1 \end{array}
\right),
$$
where $B_0= A_0 +\delta X_0.$

We begin by computing $\sigma_j(B_0),$ $1\le j \le m.$
Expanding with respect to the first column, we see
that
\begin{equation}
\label{det}
\det (tI-B_0)=t^m+(-1)^{m-1}\delta^{m-r}.
\end{equation}
The $m$ distinct roots $\lambda_1, \lambda_2, \cdots, \lambda_m$
of this polynomial are the eigenvalues of $B_0,$
and the multiplicity of $\lambda_j$ as a zero of the
minimal polynomial of $B_0$ is $1, 1 \le j \le m.$
Comparing the respective coefficients of both sides, it follows that
\begin{equation}
\label{delta}
\sigma_j(B_0)=\left\{\begin{array}{ll}
0,&1\le j\le m-1\\
\delta^{m-r},&j=m\end{array}.\right.
\end{equation}

Let $sp (A_1)=\{\mu_1, \mu_2, \cdots, \mu_s\}$,
with the multiplicity of $\mu_j$ as a zero of the minimal polynomial of $A_1$
denoted by $m_j, 1 \le j \le s.$

Consider now $\varphi \in \O(\D,\Omega_n)$ and $\zeta \in \D$ such that
$\varphi(0)=A, \varphi(\zeta)=B.$
Then, by applying \eqref{equationofBharali}, we obtain that
$$
|\zeta| \ge \max\{|\lambda_1 \cdots \lambda_m\mu_1^{m_1}
\cdots \mu_s^{m_s}|;|\lambda_j|^k \prod_{i=1}^s 
l_{\mathbb D}(\lambda_j,\mu_i)^{m_i}, 1\le j \le 
m\}.
$$
Using this and \eqref{det} with $\delta$ small enough we  have
\begin{equation}\label{1}
l_{\Omega_n}(A,B) \ge  C\cdot 
\delta^{\frac{m-r}{m} k}, \mbox{ where } C \mbox 
{ is a constant.}
\end{equation}

Take a sequence of cyclic matrices
  $A_0^j \to  A_0.$
If we consider the matrices
$$
A^j: = \left( \begin{array}{cc} A^j_0 & 0 \\ 0 & A_1 \end{array}
\right) \in \Omega_n,
$$
then $ A^j \to A$ as $j \to \infty.$

Define the map $f: \Omega_m \to\Omega_n$ by
$$
f(M)=\left( \begin{array}{cc} M & 0 \\ 0 & A_1 \end{array}
\right).
$$
Since $f(A_0^j)=A^j; f(B_0)=B,$ we have
\begin{equation}\label{2}
l_{\Omega_n}(A^j,B) \le l_{\Omega_m}(A_0^j,B_0).
\end{equation}
Since $A_0^j, B_0$ are cyclic,
we have
\begin{equation}
\label{3}
l_{\Omega_m}(A_0^j,B_0)=l_{\mathbb G_m}(\sigma(A_0^j),\sigma(B_0))
\end{equation}
Since $\mathbb G_m$ is a taut domain, $l_{\mathbb G_m}$
is a continuous function. Thus
\begin{equation}\label{4}
l_{\mathbb G_m}(\sigma(A_0^j),\sigma(B_0)) \to
l_{\mathbb G_m}(\sigma(A_0),\sigma(B_0)).
\end{equation}
On the other hand, we can find $R>0$ such that
$\mathbb B(0,R) \subset \mathbb G_m,$
where $\mathbb B(0,R)$ denotes the Euclidean ball 
with center at $0$ and radius $R.$
For $\delta$ chosen small enough, $\sigma(B_0) \in \mathbb B(0,R)$.
By the definition of Lempert function and \cite[Proposition 3.1.10]{JarnickiPflug}, we conclude that
\begin{multline}\label{5}
l_{\mathbb G_m}(\sigma(A_0), \sigma(B_0)) \le
l_{\mathbb B(0,R)}(\sigma(A_0), \sigma(B_0)) \\
= l_{\mathbb B(0,R)}((0,\cdots, 0),(0,\cdots, \delta^{m-r})) 
= \frac {\delta^{m-r}}{R}.
\end{multline}

Combining \eqref{1},\eqref{2},\eqref{3},\eqref{4} and \eqref{5}, we have
$$
l_{\Omega_n}(A,B) > l_{\Omega_n}(A^j,B)
$$
when $\delta$ is small enough and $j$ is large 
enough. It implies the discontinuity
of the Lempert function $l_{\Omega_n}(.,B)$ at the point $A.$ \qed

Note that we have proved a slightly stronger 
statement than the proposition : for $A$ to be
cyclic, it is enough that the function 
$l_{\Omega_n} (.,B)$ be continuous at $A$ for all
$B$ in some neighborhood of $A$.

\section{Proof of  Proposition \ref{contlemp}}
\label{pfnil}

\subsection{Lifting maps}
We need the following generalization of \cite[Theorem 2.8]{Agl-You1}.

\begin{prop}
\label{lift}
Let $B \in \mathcal M_n (\C)$ be a nilpotent matrix.
Then there exists a linear map $\Theta_B$ from the set of analytic maps 
from $\D$ to $\C^n$, depending only on the values at $0$  of the 
first $n-1$ derivatives
of the coordinates of the map, such that :

 Given $\zeta_0 \in \D$,
$A\in \Omega_n $ a cyclic matrix, and $ \varphi$ a holomorphic map from $\D$
to $\G_n $ such that $ \varphi(0)=\sigma(B)$, $ \varphi(\zeta_0)=
\sigma(A)$, 
then 
there exists $\tilde \varphi$ a holomorphic map from $\D$
to $\Omega_n $ such that $\sigma \circ \tilde \varphi =  \varphi$, 
$\tilde \varphi(0)=B$ and $\tilde \varphi(\zeta_0)=A$
if and only if 
$\Theta_B(\varphi)=0$. 
 \end{prop}

Notice  that in the case where $B$ is cyclic, $\sigma$ is of maximal
rank at $B$, so that $\Theta_B=0$, and conversely it is known \cite{Agl-You1}
that any map $\varphi$ will admit a lifting through cyclic matrices. 

Finally, the case where $B$ admits a single (arbitrary) eigenvalue easily reduces to the
nilpotent case 
by using the automorphism
$\Phi_\lambda (M) = (\lambda I_n - M)(I_n - \bar \lambda M)^{-1}$,
 of the spectral ball.

\begin{proof*}{\it Proof of Proposition \ref{contlemp} in the case of a single eigenvalue.}

It is enough to prove that $\ell:=\liminf_{M\to A} l_{\Omega_n}(M,B) \ge l_{\Omega_n}(A,B)$.
Suppose that $A_j\to A$ ($j\ge 1$), and that $\tilde \varphi_j \in \mathcal O(\D, \Omega_n)$
are holomorphic maps and $\zeta_j\in \D)$ such that 
$\tilde \varphi_j(0)=B$, $\tilde \varphi_j(\zeta_j)=A_j$, $\lim_{j\to\infty} |\zeta_j|=\ell$.
Consider $\varphi_j :=\sigma \circ \tilde \varphi_j$, and 
(by Montel's Theorem) extract a subsequence
converging to $\varphi \in\mathcal O(\D, \mathbb G_n)$
and such that $\zeta_j \to \zeta_0$. 

By the necessary condition in Proposition \ref{lift}, $ \Theta_B(\varphi_j)=0$;
$\varphi_j (0) = 0$ and $\varphi_j (\zeta_j) = \sigma(A_j)$.
So compact convergence implies that 
$ \varphi (0) = 0$, $\varphi (\zeta_0) = \sigma(A)$ and
$ \Theta_B(\varphi)=0$. By the sufficiency part of  Proposition \ref{lift}, there exists
$\tilde \varphi \in \mathcal O(\D, \Omega_n)$ such that 
$\tilde \varphi (0) = B$, $\tilde \varphi (\zeta_0) = A$.
Therefore $l_{\Omega_n}(A,0) \le |\zeta_0|= \ell$. 
\end{proof*}

\subsection{Proof of Proposition \ref{lift}}

It will be enough to find a map $\tilde\varphi$ that 
satisfies the conclusion only with $\tilde \varphi(0)=B'$ and $\tilde \varphi(\zeta_0)=A'$,
where $B'$, $A'$ are conjugate to $B$, $A$ respectively, as in
\cite[Proof of Theorem 2.1]{Agl-You1}. Furthermore, any cyclic matrix
with the same spectrum as $A$ will be conjugate to $A$. So it is enough
to check that $\tilde \varphi(0)=B$, where $B$ is in  Jordan form, and $\tilde \varphi(\zeta)$
is cyclic for $\zeta \neq 0$.

We use the following notations: $B=(b_{ij})_{1\le i,j\le n}$. 
For two integers $k\le l$, $[k..l]:= \{ i \in \Z : k \le i \le l\}$.
Let $r$
stand for the rank of $B$.  
Write $$
F_0:=\{ j : b_{ij}=0\mbox{ for  }1\le i \le n \}
:=\{ 1=b_1 < b_2 < \dots < b_{n-r} \}.
$$
For 
 $j \in F_1 :=[1..n] \setminus f_0$, $v_{j-1,j}=1$, $v_{ij}=0$ for $i\neq j-1$.
We can choose the Jordan form so that $b_{l+1}-b_l$ is increasing for $1\le l \le n-r$,
with the convention $b_{n-r+1}:=n+1$.

Let us give the differential conditions satisfied by a map 
$\sigma \circ \tilde \varphi$ at $0$. 
To do so, we must study the homogeneity of the functions $ \sigma_i (B+M)$ in terms of the 
entries of $M$. Let $M=(m_{kl})_{1\le k,l \le n}$.  For  $E\subset \{1, \dots, n\}$,
denote $M_E:= (m_{kl})_{ k,l \in E}$.
\begin{multline}
\label{polyinA}
 \sigma_i (B+M) = \sum_{E\subset \{1, \dots, n\}, \# E = i} \,
\mbox{det}\, (B+M)_E 
\\
= \sum_{E\subset \{1, \dots, n\}, \# E = i} \,
\sum_{\theta \in \mathcal S(E)} (-1)^{\mbox{sgn}(\theta)} \prod_{e \in E} (b_{e,\theta(e)}+m_{e,\theta(e)}),
\end{multline}
where $\mathcal S(E)$ stands for the permutation group of $E$.
This is a polynomial of degree $\le i$ in the entries of $M$. 

\begin{lemma}
\label{degree}
Let $d_i := 1+ \# \left( F_0 \cap [(n-i+2)..n] \right) $.  
The lowest order terms of  $\sigma_i (B+M) $
are of degree $d_i$ (in the entries of $M$).
\end{lemma}

\begin{proof}

The terms in this polynomial in \eqref{polyinA} are of the form
$$
\pm \prod_{e \in E_1} v_{e,\theta(e)} \cdot \prod_{e \in E_2} a_{e,\theta(e)},
$$
where $E=E_1 \cup E_2$, $E_1 \cap E_2 = \emptyset$. Non zero terms
of degree $\#E_2$ have  sets $E_1$
such that 
\begin{equation}
\label{nonzero}
v_{e,\theta(e)} = 1 \quad \forall e \in E_1,
\end{equation}
so $\theta(e) = e+1$ and $ e+1 \in F_1$. 

Write $J_k:= [(b_k+1)..(b_{k+1}-1)]$, so that $F_1=\cup_k J_k$ (some
of those may be empty), and $E_{1,k}:= \{e \in E_1 : e+1\in  J_k\}$.
If $E_{1,k} \neq \emptyset$, then $1 + \max E_{1,k} = \theta (\max E_{1,k}) \in E$,
and since it cannot be in any of the $J_l$, it must belong to $E_2$. 
Thus
$$
d:=\# E_2 \ge  \# \left\lbrace k: E_{1,k} \neq \emptyset \right\rbrace .
$$
We use the partition $E=E_2 \cup \bigcup_k E_{1,k}$:
\begin{multline*}
\#E =i = d + \sum_{k : E_{1,k} \neq \emptyset } \# E_{1,k} 
\le  d + \sum_{k : E_{1,k} \neq \emptyset } \# J_k \\
\le d + \sum_{k= n-r-d+1}^{k= n-r} (b_{k+1}-b_k-1) = (n+1) - b_{n-r-d+1}.
\end{multline*}
Therefore $b_{n-r-d+1} \le  (n+1) - i<n-i +2$, which means that 
$[(n-i+2)..n] $ contains no more elements of $F_0$ than those  
greater or equal to  $b_{n-r-d+2}$, of which there are $d-1$. So 
$
\# \left( F_0 \cap [(n-i+2)..n] \right) \le d-1 ,
$
i. e. $\# E_2 \ge d_i$. 

To show that this bound is attained, choose 
$$
E_1:= \left\lbrace  e \in [(n-i+1)..(n-1)] : e+1 \in F_1 \right\rbrace .
$$
Then $\# E_1 = i-d_i$, and we can define the permutation $\theta$
by $\theta (e) = e+1$, $n-i+1 \le e \le n$, and $\theta (n) = n-i+1$,
so that $E = [(n-i+1)..n]$ and $\# E=i$, as required. 
\end{proof}

\begin{cor}
\label{diffcond}
If $ \varphi = (\varphi_1, \dots, \varphi_n)
=\sigma \circ \tilde \varphi  $, with $\tilde \varphi  \in \O(\D,\Omega_n)$,
$\tilde \varphi (0)=B$, then $\varphi_i^{(k)}(0) = 0$, $0 \le k \le d_i -1$.
\end{cor}
This gives additional (differential) conditions whenever
$d_i \ge 2$. 

Conversely, given a map $\varphi \in \O(\D, \G_n)$ satisfying the
conclusion of Corollary \ref{diffcond}, let
$$
\psi (\zeta) :=
\left( 
\begin{array}{ccccc}
0 & f_2 & 0 & \cdots & 0 \\
0 & 0 & \ddots & & \vdots \\
\vdots & & \ddots & f_{n-1} & 0 \\
0 & 0 & \cdots & 0 & f_n \\
\psi_n & \psi_{n-1} & \cdots & \psi_2 & \psi_1 
\end{array}
\right) .
$$

\begin{lemma}
\label{pscomp}
For a matrix $\psi$ as above, $\sigma_i(\psi) = (-1)^{i+1} \psi_i \prod_{k=n-i+2}^{n} f_k$.
\end{lemma}

\begin{proof}
Expanding with respect to the last row, 
$$
\det (XI_n - M) = X^n + \sum_{j=1}^n (-1)^{j+n} (- \psi_{n+1-j}) X^{j-1}
  \prod_{k=j+1}^{n} (-f_k),
$$
and we find the coeficient of $(-1)^i X^{n-i}$ by setting $j=n-i+1$. 
\end{proof}

We define a lifting by taking $\psi$ with $f_j(\zeta) =1$
for $j \in F_1$, $f_j(\zeta) =\zeta$ for $j\in F_0$, $\psi_j(\zeta)= 
(-1)^{i+1} \zeta^{-d_j+1} \varphi_j(\zeta)$.  Those coefficients
are holomorphic, and $\psi(0)=B$, since $\varphi$ satisfies
the conclusion of Corollary \ref{diffcond}. 
Lemma \ref{pscomp} shows
that $\sigma \circ \psi = \varphi$.  Finally, for any $\zeta \neq 0$,
the matrix $\psi(\zeta)$ is cyclic, and $\sigma \circ \psi (\zeta_0)
=\sigma (A)$, so we can modify $\psi$ as indicated at the beginning
of the section, to obtain
a lifting through $A$.

\subsection{The cases $n=2$ or $3$}
\label{2or3}

The only cases to deal with are those
were $B$ is derogatory (non cyclic). As above,
we can  reduce ourselves to the case where one eigenvalue 
of $B$ is $0$.

When $n=2$, the only derogatory matrices are scalar, therefore we only
need to consider the case $B=0$ and Proposition \ref{lift} (or indeed,
 \cite[Theorem 2.8]{Agl-You1} settles the question. When $n=3$ and $B$ is derogatory,
either it has a single eigenvalue, or it has two distinct eigenvalues, one of
which has an eigenspace of dimension $2$. We then may assume that $\dim \ker B=2$, and, say,
$$
B = 
\left( 
\begin{array}{ccc}
0 & 0 &  0 \\
0 & 0 &  0 \\
0 & 0 & \lambda_1 
\end{array}
\right) ,
$$
for some $\lambda_1 \in \D$. 

If $ \varphi = \sigma \circ \tilde \varphi$, simple
determinant calculations show that 
\begin{equation}
\label{cond3}
\varphi (0) = \sigma(B)= (\lambda_1, 0, 0), \mbox{ and }
\varphi_3'(0) =0.
\end{equation}

Assume we have a map 
$\varphi \in \O (\D, \G_3)$ such that
\eqref{cond3} holds. Set
$$
\psi(\zeta) := 
\left( 
\begin{array}{ccc}
0 & \zeta &  0 \\
0 & 0 &  \zeta \\
\zeta^{-2} \varphi_3(\zeta) & -\zeta^{-1} \varphi_2(\zeta)  & \varphi_1(\zeta) 
\end{array}
\right) .
$$
Then $\sigma \circ \psi = \varphi$, 
$\psi(0)$ is conjugate to $B$ (not necessarily equal to it), and for any
$\zeta \neq 0$, $e_3$ is a cyclic vector for $\psi(\zeta) $. 

\section{Examples}
\label{exs}

Recall from \cite[Proposition 
4]{NiThZw} that when $A=tI$, the function
$l_{\Omega_n}(.,B)$ is continuous at $A$, or 
equivalently the function $l_{\Omega_n}$
is continuous at $(A,B)$, if and only if all the 
eigenvalues of $B$ are equal. For
$n=2$, this covers all the derogatory cases.
Using an automorphism of $\Omega_n$, the 
situation quickly reduces to the case $t=0$.
The next example in the case $n=3$ is then
$$
A:=
\begin{pmatrix}
0 & 0 & 0 \\
0 & 0 & 1\\
0 & 0 & 0
\end{pmatrix}.
$$

\begin{exam}
Taking
$$
B=
\begin{pmatrix}
\varepsilon & 0 & 0 \\
0 & j\varepsilon & 0\\
0 & 0 & j^2 \varepsilon
\end{pmatrix},
\mbox{ where } \varepsilon >0, j = -1/2 + i \sqrt 3 /2
$$
and $\varepsilon$ small enough, the function 
$l_{\Omega_n}(.,B)$ is discontinuous at $A$.
\end{exam}

Indeed, we clearly have $sp A=\{0\}, \sigma(A)=(0,0,0); B$ is non-derogatory.
The eigenvalues of $B$ are $\varepsilon; j\varepsilon$ and $ j^2\varepsilon.$
Thus $\sigma(B)=(0,0,\varepsilon^3).$

We can find $r>0$ such that
$
\mathbb B(0,r) \subset \mathbb G_3,
$
where $\mathbb B(0,r)$ denotes the Euclidean ball 
with center at $0$ and radius $r.$
For $\varepsilon$ chosen small enough, $\sigma(B)\in \mathbb B(0,r).$ By the definition
of Lempert function and \cite[Proposition 3.1.10]{JarnickiPflug},  we conclude that
$$
l_{\mathbb G_3}(\sigma(A), \sigma(B))=
l_{\mathbb G_3}(0, \sigma(B))\le l_{\mathbb B(0,r)}(0,\sigma(B))=
\frac{\Vert \sigma(B)\Vert}{r}=\frac {\varepsilon^3}{r}.
$$

On the other hand, if there is an analytic 
function $\varphi: \D \to \Omega_3$ such that 
$\varphi(0)=A$ and $\varphi (\zeta) = B$
then, by \eqref{equationofBharali} we have
$
\varepsilon^2 = \max\{\varepsilon^2, \varepsilon^3\} \le |\zeta|.
$
It follows that
$$
l_{\Omega_3}(A,B)\ge \varepsilon^2 >\frac 
{\varepsilon^3}{r} \ge l_{\mathbb G_3}(\sigma(A), 
\sigma(B))
$$
for $\varepsilon$ is small enough.

\begin{exam}

If the eigenvalues of $B$ are equal, then the function 
$l_{\Omega_n}(.,B)$ is continuous at $A$ (moreover
 $l_{\Omega_n}$ is continuous at $(A,B)$).
\end{exam}

Indeed, if the eigenvalues of $B$ are equal, say to $\mu$, then, by \cite{NiThZw} 
and \eqref{decr} we have 
$$
|\mu| \ge l_{\Omega_3}(A,B) \ge l_{\mathbb G_3}(\sigma(A),\sigma(B)).
$$
On the other hand, if $C_{\mathbb G_3}$ 
is the Carath\'eodory pseudodistance of $\mathbb G_3$ then
\begin{multline*}
l_{\mathbb G_3}((0,0,0),(3\mu, 3\mu^2, \mu^3))
= l_{\mathbb G_3}(\sigma(A),\sigma(B)) 
\ge C_{\mathbb G_3}(\sigma(A),\sigma(B)) \\
\ge \sup_{|\lambda|=1} 
\bigg | 
\frac{f_{\lambda}(\sigma(B))-f_{\lambda}(\sigma(A))}
{1-\overline{f_{\lambda}(\sigma(A))}f_{\lambda}(\sigma(B))}
\bigg| = |\mu|,
\end{multline*}
where $f_{\lambda}(S)=
\dfrac{s_1+2s_2\lambda + 3s_3\lambda^2}
{3+2s_1\lambda +s_2\lambda^2}, 
\forall S=(s_1, s_2, s_3)\in \mathbb G_3, \lambda \in \overline{\D},
$
(for the last inequality see \cite{costara} or \cite{NiPfZw}).
Thus
$
l_{\Omega_3}(A,B) = l_{\mathbb G_3}(\sigma(A),\sigma(B)) = |\mu|.
$

\section{The Pluricomplex Green and Lempert functions are not equal} 
\label{green}

Let $D$ be a domain in $\C^n$. We call Green function the exponential
of the usual Green function with logarithmic singularity.

\begin{defn}
For $(a,z) \in D\times D$,
the pluricomplex Green function with pole at $a$,
evaluated at $z$ is
\begin{multline*}
g_D(a,z):= \sup\{u(z): u: D \to [0,1), \log u \in PSH(D), \\
\exists C=C(u,a)>0, \forall w \in D: u(w)\le C\|w-a\|\}, 
\end{multline*}
where $PSH(D)$ denotes the family of all functions plurisubharmonic on $D$ 
(and $\| \, \, \|$ is the Euclidean norm in $\C^n$).
\end{defn}

The formulas for the Carath\'eodory and the Lempert functions 
on $\mathbb G_2$ were obtained 
by Agler and Young \cite{agleryoung2}. 
Using the fact that $C_{\mathbb G_2}=l_{\mathbb G_2},$
Costara \cite{costara2} has obtained a formula for the 
Carath\'eodory and the Lempert functions on $\Omega_2.$
He proved that on $\Omega_2$ the Carath\'eodory and the Lempert
functions do not coincide. The same holds for the Green function,
for any $n$.

\begin{prop}
Let $A$ be a cyclic matrix in $\Omega_n$ 
such that at least two of the eigenvalues of $A$ are not equal.
Then there exists a matrix $B\in \Omega_n$ such that
$$
l_{\Omega_n}(A,B) >g_{\Omega_n}(A,B).
$$
\end{prop}

For  Green functions with several poles, 
there are cases of strict inequality even for the usual bidisk \cite{TraoTh}.

\begin{proof}
We need the following.

\begin{prop}[Edigarian-Zwonek] 
Let $A,B \in \Omega_n.$ Then 
\begin{equation}
\label{Green}
g_{\Omega_n}(A,B) =g_{\Omega_n}(A, diag(\mu_1, \mu_2, \cdots, \mu_n)),
\end{equation}
where $sp(B) = \{ \mu_j , 1 \le j \le n\}$, with the eigenvalues repeated
according to multiplicity. 
\end{prop}

Take $B_{\alpha}=(b_{i,j})_{1\le i, j \le n}\in \Omega_n,$ where 
$b_{1,1}=b_{2,2}= \cdots = b_{n,n}=\mu \in \D; b_{j-1,j}=\alpha \in \C, \forall 2\le j \le n,$ 
and all other coefficients $b_{i,j}=0, 1\le i,j\le n$.

Since $B_0$ is a scalar matrix, by using \cite{zerosetofkobayashi}, 
we have discontinuity of the Lempert function $l_{\Omega_n}$ at $(A,B_0).$
Then by Proposition \ref{contcrit}
$$
l_{\Omega_n}(A,B_0) > l_{G_n}(\sigma(A),\sigma(B_0)).
$$

Consider now $\alpha \not= 0.$ Then $A, B_{\alpha}$ are cyclic matrices. It implies that 
$$
l_{G_n}(\sigma(A),\sigma(B_0)) = l_{G_n}(\sigma(A),\sigma(B_{\alpha})) = l_{\Omega_n}(A,B_{\alpha}).
$$
On the other hand 
$$
l_{\Omega_n}(A,B_{\alpha}) \ge g_{\Omega_n}(A,B_{\alpha}) = g_{\Omega_n}(A,B_0),
$$
where the last equality follows from \eqref{Green}.

Thus $l_{\Omega_n}(A,B_0) >g_{\Omega_n}(A,B_0).$

\end{proof}

\bibliographystyle{amsplain}

\end{document}